\documentclass[12pt, a4paper]{amsart}
\usepackage{amsfonts,amsmath, amsthm, amssymb, amscd}
\usepackage{latexsym}
\input{amssym.def}
\input{amssym}

\newtheorem{defi}{Definition}[section]
\newtheorem{satz}[defi]{Theorem}
\newtheorem{prop}[defi]{Proposition}

\newtheorem{lemma}[defi]{Lemma}
\newtheorem{bem}[defi]{Remark}
\newtheorem{folg}[defi]{Corollary}

\newcommand {\R}{\mathbb{R}} 
\newcommand {\C}{\mathbb{C}} 

\DeclareMathOperator{\vol}{vol}

\begin{document}
\title{The Brownian Motion on Harmonic Manifolds of Purely Exponential Volume Growth}

\author{Oliver Brammen}
\address{Faculty of Mathematics,
Ruhr University Bochum, 44780 Bochum, Germany}
\email{oliver.brammen@rub.de}
\thanks{The author would like to thank Kingshook Biswas for a hint at this question and his explanation of Sullivan's construction. Furthermore, the author would like to thank Gerhard Knieper and Norbert Peyerimhoff for helpful discussions and encouragement.  The author is supported by the German Research Foundation (DFG), CRC TRR 191, Symplectic structures in geometry, algebra and dynamics.}


\begin{abstract}
We show that Sullivan's method of constructing eigenfunctions of the Laplacian \cite{Sullivan1987RelatedAO}  is applicable to non-compact simply connected harmonic manifolds of purely exponential volume growth. Thereby we extend the case of pinched negative curvature to non-symmetric  Damek-Ricci spaces.
\end{abstract}


\maketitle

\section{Introduction}
 In the 1820's the English botanist Robert Brown first published on the random behaviour of pollen suspended in water. Later this behaviour was named Brownian motion and first explained by Albert Einstein \cite{einstein1905erzeugung} to originate from random collisions with water molecules. Which in turn results from the thermal movement of the molecules. 
The first rigorous mathematical description was provided by Norbert Wiener \cite{wiener1923differential}. Therefore it is also known as Wiener process. 
We can characterise the Brownian motion by whether a partial will return to every bounded region infinitely often or escape it forever. 
The first one is called recurrent the second transient. 
Sullivan in \cite{Sullivan1987RelatedAO} used the transients of the Brownian motion on negatively curved manifolds to construct eigenfunctions of the Laplacian. We will use the same method to construct eigenfunctions of the Laplacian on pre-compact regions with smooth boundaries inside a harmonic manifold of rank one.
For this purpose, we will give a construction of the Brownian motion using the heat semi-group and then proceed to adapt the arguments of Sullivan to our framework.
Note that a comprehensive proof of Sullivans result can be found in  \cite{biswas2019sullivans} and the transient behaviour of the Brownian motion on manifolds with exponential volume growth is well known. See for instance  \cite[Theorem 5.1 and Chapter 10,11]{Grigor99}. We adapted and amended these arguments for harmonic manifolds, where the main difference is that in the pinched negative curvature case, the volume of a geodesic ball of a fixed radius is uniformly bounded whereas in the harmonic case, the volume is the same. Furthermore, we require some more subtle arguments from \cite{Maheux1995}  to obtain the Poincare inequality and thereby exponential decay of the mass of the heat kernel outside geodesic ball. Apart from these modifications the proof is identical and included for the reader's convenience.  This extends the results for pinched negative curvature to Damek Ricci spaces where the curvature is only non-positive. 
This article is structured as follows: In section 2 we give a brief overview of the construction of Markov process and the properties of harmonic manifolds with a special emphasis on the heat and Green kernel. In section 3 we verify that the heat semi-group is feller and therefore generates the Brownian motion. In section 4 we then use the expected exit time of the Brownian motion from pre-compact domains to construct eigenfunctions of the Dirichlet-Laplacian on those domains. After briefly recapitulating well-known facts about the heat semi-group and its connection to the Brownian motion.

\section{Preliminaries}
\subsection{Markov Process}
Let $(M,d)$ be a locally compact, separable metric space equipped with the Borel sigma algebra $\mathcal{B}(M)$. Let $B(M)$ denote the space of bounded measurable functions equipped with the norm of uniform convergence $\lVert \cdot\rVert_\infty$  and $C_0(M)$ the space of continuous functions vanishing at infinity contained within $B(M)$. 
We begin by recalling the definition of the Brownian motion and associated concepts. 
Let $(\Sigma,\mathcal{F},\mathbb{P})$  be a probability space and $\psi\in L^1(\Sigma,\mathcal{F},\mathbb{P})$. Given any sigma algebra $\mathcal{G}\subset \mathcal{F}$ the conditional expectation for $\psi\in L^1(\Sigma,\mathcal{F},\mathbb{F})$ 
is the unique $\mathcal{G}$-measurable random variable $\mathbb{E}(\psi\mid\mathcal{G})\in  L^1(\Sigma,\mathcal{G},\mathbb{P})$ such that
\begin{align}
 \int_{A}\mathbb{E}(\psi\mid\mathcal{G})\,d\mathbb{P}=\int_{A}\psi \,d\mathbb{P}\quad \forall A\in\mathcal{G}.
 \end{align}

\begin{lemma}
$\mathbb{E}(\psi\mid\mathcal{G})$ exists and is unique.
\end{lemma}
\begin{proof}
For $A\in \mathcal{G}$ let $\mu(A)=\int_A \psi d\mathbb{P}$. By the Radon-Nikodym theorem there exists a function $f:\Sigma\to X$ which is measurable with respect to $\mathcal{G}$ such that 
\begin{align*}
\mu(A)=\int_A f \,d\mathbb{P}.
\end{align*}
Hence, $f$ is the conditional expectation and is by definition unique. 
\end{proof}
 Let $y:\Sigma\to M$ be a random variable. The sigma-algebra generated by $y$ is denoted by $\sigma(y)$ 
and the conditional expectation of a random variable $\psi:\Sigma\to M$ with respect to $y$ is given by
\begin{align*}
\mathbb{E}(\psi\mid y):=\mathbb{E}(\psi\mid\sigma(y)).
\end{align*}
For $\mu$ a probability measure on $M$, a continuous Markov process on $M$ with initial distribution $\mu$ is a collection $(B_t)_{t\geq 0}$ of $M$-valued random variables on $(\Sigma,\mathcal{F},\mathbb{P})$ such that for all $f\in B(M)$ we have:
\begin{align*}
\mathbb{E}(f(B_t)\mid \mathcal{F}_s)&=\mathbb{E}(f(B_t)\mid B_s)\quad \forall t\geq s,
\end{align*}
where
\begin{align*}
\mathcal{F}_s&=\bigcup\limits_{u=0}^{s}\sigma(B_u)
\end{align*}
and 
\begin{align*}
\mathbb{P}(B_0\in E)=\mu(E)\quad \forall E\in \mathcal{B}(M).
\end{align*}
%
A sample path of the Markov process $B_t$ is a map $t\to B_t(\omega)$ for some $\omega\in \Sigma$. 


Finally let   $(T(t))_{t\geq 0}$ be  a semi-group of bounded operators on $B(M)$. A Markov process corresponds to the semi-group $(T(t))_{t \geq 0}$ if for all $f\in B(M)$ and $s,t\geq 0$ we have:
\begin{align*}
\mathbb{E}(f(B_{s+t}\mid \mathcal{F}_s))=(T(t)(f))(B_s).
\end{align*}
Such a semi-group $(T(t))_{t\geq 0}$ is called Feller-semi-group if it satisfies the following conditions:
\begin{enumerate}
\item The semi-group is contracting, that is:   $\lVert T(t)\rVert \leq 1 $ for all $t\geq 0$ where, $\lVert\cdot \rVert$ denotes the operator norm.
\item The semi-group is positive: $ T(t)(f)\geq 0 ~\forall f\geq 0$.
\item The semi-group fixes the space of bounded continuous functions on $M$ denoted by $C_0(M)$, also called Feller property, or more precisely:  For all $f\in C_0(M)$ and for all $t\geq 0$ we have  $T(t)(f)\in C_0(M)$.
\item The semi-group is strongly continuous: For all $f\in C_0(M)$ we have $\lVert T(t)f-f\rVert_{\infty}\to 0$ as $t$ approaches $0$.
\item The semi-group is conservative: $T(t)1_M=1_M$.
\end{enumerate}
The following theorem is a classical result which can be found in \cite{ethier2009markov} and guarantees the existence of general Markov processes given a Feller semi-group and an initial distribution.
\begin{satz}[\cite{ethier2009markov}]\label{thm:markov}
Let $(T(t))_{t\geq 0}$ be a Feller semi-group operating on $B(M)$ and $\mu$ a probability measure on $M$. Then there is a Markov process $(B_t)_{t\ge0}$ with initial distribution $\mu$ which corresponds to the semi-group $(T(t))_{t\geq 0}$ such that for $\mathbb{P}$-almost every $\omega \in \Sigma$ the curve $t\to B_t(\omega)$ is  right continuous and its left limit exists for all $t$.
\end{satz}
Note that the theorem gives the existence of the probability space $(\Sigma,\mathcal{F},\mathbb{P})$.

In the next step we want to define the transition probability of such a Markov process. 
Let $(\Sigma,\mathcal{F},\mathbb{P})$ be the Markov process given by  Theorem \ref{thm:markov} corresponding to the initial distribution 
$\delta_x$. Then all sample paths are and their left limit exists for all $t$. If $D_M([0,\infty))$ denotes the space of all  such paths,
then the map 
\begin{align*}
\Psi:\Sigma\to D_M([0,\infty))\\
\Sigma\ni\omega\mapsto (t\mapsto B_t(\omega)) 
\end{align*}
is defined for $\mathbb{P}$ -almost every $\omega\in \Sigma$. We obtain a probability space with sample space $D_M([0,\infty))$ by pulling back $\mathcal{B}(M)$ under the maps $\pi_t:D_M([0,\infty))\to M$ which evaluates a path $\gamma$ at $t$. Hence we get a sigma algebra $D=\bigcup_{t\geq 0} \pi_t^{-1}(\mathcal{B}(M))$ and a probability measure $\mathbb{P}_{*}=(\Psi)_{*}\mathbb{P}$. This in turn gives a probability measure $p(t,x,\cdot)$ on $M$ defined by 
\begin{align*}
p(t,x,\cdot)=(\pi_t)_{*}\mathbb{P}_{*}.
\end{align*}
 Since $\pi_t\circ \Psi=B_t$, the distribution of $B_t$ is given by $p(t,x,\cdot)$.
We call the measure $p(t,x,\cdot)$ the transition probability associated to the semi-group $(T(t))_{t\geq 0}$.
By the definition of the Markov process and the fact that $B_t$ is the Markov process associated to $(T(t))_{t\geq 0}$ with initial distribution being the delta distribution at $x\in M$ we get:
\begin{align*}
\int_{\Sigma}\mathbb{E}(f(B_t)\mid \mathcal{F}_0)\,d\mathbb{P}=\int_\Sigma(T(t)f)(B_0)\,d\mathbb{P}.
\end{align*}
Since $B_t$ has the delta distribution as initial distribution:
\begin{align*}
 \int_{\Sigma}f(B_t)\,d\mathbb{P}=T(t)f(x).
 \end{align*}
By  $\pi_t\circ \Psi=B_t$ and the discussion above we finally obtain:
\begin{align*}
\int_M f(y)\,d p(t,x,y)=(T(t)f(x)).
\end{align*}
The theorem below is again a classical result and can be found in \cite{ethier2009markov}.
\begin{satz}[\cite{ethier2009markov}]\label{thm:conpath}
If for every $x\in M$ and $\epsilon>0$ we have
$$\lim\limits_{t\to 0} \frac{1}{t} p(t,x,M\setminus B(x,\epsilon))=0,$$
where $B(x,\epsilon)$ is the metric ball of radius $\epsilon$ around $x$,
then almost all sample paths are continuous. 
\end{satz}
Hence, under the condition of Theorem \ref{thm:conpath} we obtain for each $x\in M$ a probability measure $\mathbb{P}_x$ on the space of all continuous paths $\gamma:[0,\infty)\to M$ denoted by $C_M$ and a sigma algebra $\mathcal{C}$ generated by the map
 \begin{align*}
\pi_t: C_M&\to M\\
\gamma&\mapsto \gamma(t).
\end{align*}
This consideration allows us to focus on the space $(C_M,\mathcal{C},\mathbb{P}_x)$ and disregard the abstract probability space provided by Theorem \ref{thm:markov}. 
\subsection{Harmonic Manifolds of Purely Exponential\\ Volume Growth}
In this section, we give a brief introduction to non-compact simply connected harmonic manifolds.
 For more information 
 we refer the reader to the surveys \cite{kreyssig2010introduction} and \cite{knieper22016}.
 Let $(X,g)$ be a non-compact simply connected Riemannian manifold without conjugate points. Denote by $C^k(X)$ the space of $k$-times differentiable functions on $X$ and by $C^k_c(X)\subset C^k(X)$ those with compact support. With the usual conventions for smooth and analytic functions. We equipp thease spaces with the usual topologies see  \cite[Chapter II Section 2]{helgason2000groups}.
 Furthermore, for $p\geq 1$ $L^p(X)$ refers to the $L^p$-space of $X$ with regards to the measure induced by the metric and integration over a manifold is always interpreted as integration with respect to the canonical measure on this manifold unless stated otherwise.
   For $p\in X$ and $v\in S_pM$ denote by $c_v :\R\to M$ the unique unit speed geodesic with $c(0)=p$ and $\dot{c}(0)=v$. Define $A_v$ to be the Jacobi tensor along $c_v$ with initial conditions $A_v(0)=0$ and $A^{\prime}(0)=\operatorname{id}$. For details on Jacobi tensors see \cite{KNIEPER2002453}. Then using the transformation formula and the Gauss lemma the volume of the sphere of radius $r$ around $p$ is given by:
\begin{align}\label{eq:volS}
\operatorname{vol}S(p,r)=\int_{S_p M}\operatorname{det} A_v(r)\,dv.
\end{align}
The second fundamental form of $S(p,r)$ is given by $A_v^{\prime}(r)A^{-1}_v(r)$ and the mean curvature by 
\begin{align}\label{eq:meancurv}
\nu_p(r,v)=\operatorname{trace}A_v^{\prime}(r)A_v^{-1}(r).
\end{align}

\begin{defi}
Let $(X,g)$ be a complete non-compact simply connected manifold without conjugate points and $SX$ its unit tangent bundle. For $v\in SX$ let $A_v(t)$ be the Jacobi tensor with initial conditions $A_v(0)=0$ and $A^{\prime}_v(0)=\operatorname{id}$. Then $X$ is said to be harmonic if and only if
$$A(r)=\operatorname{det}(A_v(r))\quad\forall v\in SX.$$
Hence the volume growth of a geodesic ball centred at $\pi(v)$ only depends on its radius.
\end{defi}
From  (\ref{eq:meancurv}) one easily concludes that the definition above is equivalent to the mean curvature of geodesic spheres only depending on the radius. More precisely the mean curvature of a geodesic sphere $S(x,r)$  of radius $r$ around a point $x\in X$ is given by $\frac{A'(r)}{A(r)}$. 

Using $A_v$ one can construct the Jacobi tensor $S_{v,r}$ along $c_v$  with $S_{v,r}(0)=\operatorname{id}$, $S_{v,r}(r)=0$, and $U_{v,r}=S_{v,-r}$.
Then the stable respectively unstable Jacobi tensor is obtained via the limiting process: 
\begin{align*}
S_v&=\lim\limits_{r\to\infty}S_{v,r}\\
U_v&=\lim\limits_{r\to \infty} U_{v,r}.
\end{align*}
Note that these limits exist \cite{KNIEPER2002453}.
Let $v\in S_p X$ and $c_v$ be the unit speed geodesic with initial direction $v$.
Now define for $x\in X$ the Busemann function $b_v(x)=\lim_{t\to \infty}b_{v,t}(x)$, where $b_{t,v}(x)=d(c_v(t),x)-t$.
This limit exists and is a $C^{1,1}$ function on $X$, see for instance \cite{rub.181283719860101}. 
The level sets of the Busemann functions, $H^s_{v}:=b^{-1}_v(s)$ are called horospheres and in the case that $b_v\in C^2(X)$ their second fundamental form in $\pi(v)=p$ is given by $U_v^{\prime}(0)=:U(v)$. Hence their mean curvature is given by the trace of $U(v)$. In the case of a harmonic manifold $v\to \operatorname{trace}U(v)$ is independent of $v\in SX$, hence the mean curvature of horospheres is constant. 
Using this notion of stable and unstable Jacobi tensors Knieper in \cite{knieper2009new} generalised the well-known notion of rank for general spaces of non-positive curvature introduced by Ballmann, Brin and Eberlein \cite{Ballmann1985} 
 to manifolds without conjugated points. 
 Define for $v\in SX$  $S(v):=S'_v(0)$ and $D(v)=U(v)-S(v)$. Then:
\begin{align*}
\mathcal{L}(v)&:=\operatorname{Kern}(D(v))\\
\operatorname{rank}(v)&:=\operatorname{dim}\mathcal{L}(v)+1\\
\operatorname{rank}(X)&:=\operatorname{min}\{\operatorname{rank}(v)\mid v\in SX\}.
\end{align*}
Furthermore, Knieper showed that  for a non-compact harmonic manifold the following are equivalent: 
\begin{itemize}
\item $X$ has purely exponential volume growth meaning: 
 exists some constant $C\geq1$ and $\rho>0$ such that:
\begin{align*}
\frac{1}{C}\leq \frac{A(r)}{e^{2\rho r}}\leq C.
\end{align*}
\item the geodesic flow is Anosov with respect to the Sasaki metric
\item Gromov Hyperbolicity 
\item Rank one.
\end{itemize}

\subsubsection{Heat Kernel}
The notion of harmonicity is Furthermore equivalent to the existence of radial solutions of the Laplace equation on a punctured neighbourhood around each point in the manifolds. Recall that 
for $f\in C^2(X)$ the Laplace-Beltrami operator is defined by
\begin{align*}
\Delta f:=\operatorname{div}\operatorname{grad} f
\end{align*}
 $\Delta$ is by definition linear on $C^{\infty}_c(X)$ and we have 
\begin{align}\label{eq:deltagrad}
\int_X -\Delta f(x)\cdot f(x) \,dx=\int_X\lVert \nabla f(x)\rVert_g^2 \,dx\quad\forall f\in C^{\infty}_c(X)
\end{align}
where $\lVert\cdot\rVert_g$ is the norm induced by $g$. 
Hence $-\Delta$ is a non-negative symmetric operator. Furthermore $-\Delta$ is formally self adjoint hence by the density of $C^{\infty}_c(X)$ in $L^2(X).$ 
Therefore $\Delta$ extents to a densely defined self-adjoint operator on $L^2(M)$ with domain $D(\Delta)$. By abuse of notation, this extension will again be denoted by $\Delta$. The above also implies that the essential spectrum of  $\Delta$ is contained in the negative half line and all eigenfunctions are analytic since harmonic manifolds are Einstein and therefore analytic in normal coordinates. 
Furthermore, for $t>0$ the Laplacian $\Delta$ generates a semi-group $e^{t\Delta}$ on $L^2(X)$, defined by the spectral theorem for unbounded self-adjoint operators, which is analytical and strongly continuous.  $e^{t\Delta}$ leaves the set $L^1(X)\cap L^{\infty}(X)\subset L^2(X)$ invariant and is 
positive and contracting on $L^{\infty}(X)$, therefore extends to a positive, contracting and strongly continuous semi-group on $L^p(X)$ for $1\leq p< \infty$. By abuse of notation, it will be denoted by $e^{t\Delta}$ and is called the heat semi-group. This abuse of notation is justified by the fact that the heat semi-group is consistent: for $f\in C^{\infty}_c(X)$ the function $u:X\times\R_{\geq 0}\to\C$ given by $u(x,t)=e^{t\Delta}f(x)$ is the unique solution of the heat equation:
\begin{align*}
u:X&\times\R_{\geq 0}\to\C\\
\frac{\partial}{\partial t} u&=\Delta u,\\
u(x,0)&=f(x).
\end{align*}
The integral kernel of this action is called the heat kernel. In \cite{szabo1990} the author showed that $(X,g)$ being harmonic is equivalent to the heat kernel being a radial function and the action of the heat semi-group is defined by the convolution with the heat kernel
\begin{align*}
h_{t}:X\times X&\to \R\\
\frac{\partial}{\partial t} h_{t}&=\Delta h_{t},\\
\lim_{t\to 0} h_{t}(x,y)&=\delta_{x}(y).
\end{align*}
Where the convolution on harmonic manifolds is defined in the following way:
Let $g:X\to \C$ be a radial function  around $\sigma\in X$, with $g=u\circ d_\sigma$ for some function $u:\R_{\geq 0}\to \C$. 
 Then the convolution of a function $f:X\to \C$ with compact support with $g$
   is given by $$ (f*g)(x):=\int_{X} f(y)\cdot\tau_x g(y)\,dy.$$
Now we need to check that the integral above is defined for almost every $x\in X$, if $g$ and $f$ are in $L^1(X)$.  For that purpose let $g=u\circ d_{\sigma}$. Then we have:
\begin{align*}
\lVert f*g\rVert_1&\leq \int_X\int_X\lvert f(y)\rvert\lvert(\tau_xg)(y)\rvert \,dy\,dx\\
&=\omega_{n-1}\int_X\lvert f(y)\rvert\Bigl(\int_{0}^{\infty}\lvert u(r)\rvert A(r) \,dr\Bigr) \,dy\\
&=\lVert f\rVert_1\lVert g\rVert_1< \infty. 
\end{align*}
Moreover, by replacing $f\in L^1(X)$ by $f\in L^{\infty}(X)$ in the calculation above we obtain:
\begin{align}\label{eq:conalg}
\lVert f*g\rVert_{\infty}\leq \lVert f\rVert_{\infty}\cdot \lVert g\rVert_1.
\end{align}
Hence, by the Riesz-Thorin theorem (see for instance\cite{bennett1988interpolation}), we get for 
 for all $ p\in[1,\infty]$ and $f\in L^p(X),$
\begin{align*}
\lVert f*g\rVert_p\leq \lVert f\rVert_{p}\cdot\lVert g\rVert_1.
\end{align*}
Since $h_{t}\in L^{1}(X)$ and radial we obtain that for $f\in L^{p}(X)$ 
$$ e^{\Delta t}f(y)=(f*h_{t})(y)=\int_{X}f(x)\cdot h_{t}(x,y)\,dx.$$
\subsubsection{Green Kernel}
The function 
\begin{align*}
G:X\times X\setminus \{(x,x)\mid x\in X\} \to \R
\end{align*}
 given by 
 \begin{align}
G(x,y)=\int_{0}^{\infty}h_t(x,y)dt
\end{align}
is called the Green kernel.
$G$ is defined uniquely by the following properties:
\begin{enumerate}
\item $\Delta_x G(x,y)=0\quad \forall x\neq y$,
\item $G(x,y)\geq 0\quad \forall x\neq y$,
\item For all $y\in X$ we have $\inf_{x\in X,x\neq y}G(x,y)=0$,
\item $\int_X G(x,y)\Delta \psi(y)\,dy=-\psi(x)\quad \forall \psi\in C_0^{\infty}(X)$.
\end{enumerate}
 Since the heat kernel is a radial function, so is the Green kernel. It was shown in 
\cite{knieper2013noncompact2}  that the Green kernel can be represented using the density function $A(r)$ by:
 \begin{align}\label{green}
 G(x,y)=\frac{1}{\vol(S_yX)}\int_{d(x,y)}^{\infty}\frac{1}{A(r)}\,dr.
 \end{align}


 From now on assume that $(X,g)$ is a non-compact simply connected harmonic manifold with mean curvature of the horosphere $h=2\rho$. In this case, the authors showed in \cite{BusemannHarmonic} that $\Delta b_v=h$ and hence the Busemann functions as well as all eigenfunctions of $\Delta$ are analytic by elliptic regularity since harmonic manifolds are Einstein, see for instance  \cite[Sec. 6.8]{willmore1996riemannian}, and therefore analytic by the Kazdan-De Truck theorem \cite{ASENS_1981_4_14_3_249_0}. Furthermore, the authors in \cite[Corollary 5.2]{PS15} showed that the top of the spectrum of $\Delta$ is given by $-\rho^2$.
 Additionally, we will also from now on assume that $X$ is of purely exponential volume growth, in this case, the constant in the exponent coincides with $\rho$.

\section{Feller Property of The Heat semi-group on Harmonic Manifolds of Purly Exponential Volume Growth}
In this section, we are going to show that the heat semi-group on a rank one harmonic manifold is Feller. Let us start with three essential lemmas. We are again going to assume that $(X,g)$ is a non-compact simply connected harmonic manifold of rank one with dimension bigger than 2. 
By \cite[Thm 8.5]{chavel1984eigenvalues} we have the following lemma:
\begin{lemma}\label{complete}
$X$ is stochastic complete, that is $\Vert h_t\Vert_{1}=1$.
\end{lemma}
The second lemma is an adaptation of the more general case found in \cite[Proposition1.1]{grigor1994heat} which can be applied since the Ricci curvature of $(X,g)$ is bounded from below and all balls of the same radius have the same volume. 
To see this we need the following theorem from \cite{Maheux1995}:
\begin{satz}[\cite{Maheux1995}]
Let $(M,g)$ be a complete connected Riemannian manifold of dimension $n$, such that the Ricci curvature is bounded from below by $-a^2g$ for $a\geq 0$. Then there exists a constant $c_n$, only depended on $n$ and the tuple $(p,q)$ satisfying $1\leq p<2$ and $p \leq q \leq \frac{pn}{(n-p)}$, and a constant $C_{p,q}$ such that for every $x\in M$ and $r>0$ we have for $f$ a smooth function on the ball $B(x,r)$:
\begin{align*}
C_{p,q}e^{c_n(1+ar)}r\vol(B(x,r))^{1/q-1/p}\Big(\int_{B(x,r)}\lvert\operatorname{grad} f(y)\rvert^p\,dy\Big)^{1/p}\\\geq\Big(\int_{B(x,r)}(f(y)-\overline{f}(x))^q\,dy\Big)^{1/p},
\end{align*}
where $\overline{f}=\frac{1}{\vol(B(x,r))}\int_{B(x,r)}f(y)\,dy.$
\end{satz}
\begin{lemma}\label{lemma:heateastimate4}
Let $X$ be a non-compact simply connected harmonic manifold of rank one and $\lambda_1=-\rho^2$. Then for the heat kernel $h_t(x,y)$ there exist $C>0$ and $D>4$ such that off the diagonal in $X\times X$  for $t>0$ we have:
$$h_t(x,y)\leq \frac{C}{\min\{1,t^{n/2}\}}\exp\Bigl(\lambda_1t-\frac{d(x,y)^2}{Dt}\Bigr).$$
Furthermore there is a constant $K>0$ such that for all $x,y$ off the diagonal in $X\times X$  for $t>1$ we have:
$$h_t(x,y)\leq K\Bigl (1+\frac{d(x,y)^2}{t}\Bigr)^{1+\frac{1}{n}}\exp\Bigl(\lambda_1t-\frac{d(x,y)^2}{4t}\Bigr).$$
\end{lemma}

\begin{proof}
For the convenience of the reader, we will give a short exposition of the arguments to derive the estimates above.
By \cite[Proposition1.1]{grigor1994heat} the estimates of the lemma are valid if there are positive constants $a,b,c$ and $d$ such that  for any $x\in X$ the following properties hold:
\begin{enumerate}
\item For $r<R<1$ we have 
\begin{align*}
\frac{\vol(B(x,R))}{\vol(B(x,r))}\leq\Bigl(\frac{R}{r}\Bigr)^{d}.
\end{align*}
\item For any smooth function $f$ in the Ball $B(x,R)$ of radius $R<1$ we have:
\begin{align}\label{eq:Poincare inequality}
\int_{B(x,R)}\lvert\operatorname{grad} f(y)\rvert^2\,dy\geq\frac{b}{R^2}\int_{B(x,R/2)}(f(y)-\overline{f}(x))^2\,dy,
\end{align}
where $\overline{f}=\frac{1}{\vol(B(x,R/2))}\int_{B(x,R/2)}f(y)\,dy.$
\item $\vol(B(x,1))\geq c$.
\end{enumerate}
(i) and (iii) are obviously true on harmonic manifolds of rank one, since the volume of the ball does not depend on the centre.
(ii) is true for all manifolds with Ricci curvature bounded from below and dimension bigger then 2 see  \cite[Theorem 1.1]{Maheux1995} above. Where one chooses $p=q=2$ and uses that 
$$\int_{B(x,R)}\lvert\operatorname{grad} f(y)\rvert^2\,dy\geq \int_{B(x,R/2)}\lvert\operatorname{grad} f(y)\rvert^2\,dy.$$
\end{proof}
Note that (\ref{eq:Poincare inequality}) is called  Poincare inequality.
For the remainder of this chapter, we require further estimates which highlight why the rank one assumption in Theorem  \ref{thm:feller}, Theorem \ref{thm:trans} and  Theorem \ref{thm:exitbrown} is necessary. 

\begin{lemma}\label{expdecay}
Let $t\geq 0$. Then for $R\geq (D2\rho t)$ and  $x\in X$ there are constants $\eta,\kappa$ depending on $t$ such that:
$$\int_{X\setminus B(x,R)}h_t(x,y)\,dy\leq \kappa e^{-\eta R}.$$
\end{lemma}
\begin{proof}
With Lemma  \ref{lemma:heateastimate4} and the purely exponential volume growth of $X$ we get:
\begin{align*}
\int_{X\setminus B(x,R)}h_t(x,y)\,dy&\leq \frac{C}{\min\{1,t^{n/2}\}}e^{\lambda t}\int_{R}^{\infty}e^{\frac{-r^2}{Dt}}A(r)\,dr\\
&\leq \frac{cC}{\min\{1,t^{n/2}\}}\int_{R}^{\infty}e^{-r^2/Dt}\cdot e^{2\rho r}\,dr\\
&\leq\frac{cC\sqrt{Dt}}{\min\{1,t^{n/2}\}}\int_{\frac{R}{\sqrt{Dt}}-\frac{\sqrt{Dt}2\rho}{2}}^{\infty}e^{-s^2}\,ds,
\end{align*}
where the lower bound of the integral is bigger than zero by our assumptions on $R$.
Now by the bounds on the Gaussian integral, there is a $C_0>0$ such that:
\begin{align*}
\int_{X\setminus B(x,R)}h_t(x,y)\,dy\leq\frac{C_0\sqrt{t}}{\min\{1,t^{n/2}\}}e^{-R^2/4Dt}.
\end{align*}
Since the exponential decay is dominant, we can choose $\kappa$ large enough and $\eta$ small enough to obtain the assertion. 
\end{proof}

\begin{satz}\label{thm:feller}
 Let $X$ be a non-compact simply connected harmonic manifold of rank one.
The heat semi-group $(e^{t\Delta})_{t\ge0}$ is Feller and the paths of the associated Markov process called the Brownian motion $(B_t)_{t\geq0}$ with probability space $(\Sigma,\mathcal{F},\mathbb{P})$ is continuous for all $\omega\in\Sigma$.
\end{satz}
\begin{proof}
Let us check the conditions one by one.
\begin{description}
\item[Contraction]
The action of the heat semi-group is given by convolution with the radial function $h_t$. By applying (\ref{eq:conalg}) we get for $f\in C_0^{\infty}(X)$
\begin{align*}
\lVert e^{t\Delta}f\rVert_{\infty}\leq \lVert h_t\rVert_{1}\cdot \lVert f\rVert_{\infty}.
\end{align*}
Since in Lemma \ref{complete}  we proved that $X$ is stochastic complete, we have $\lVert h_t\rVert_{1}=1$. Hence, the assertion follows.
\item[Positivity]
This follows from the positivity of $h_t$.
\item[Feller property]
Let $f\in C_0^{\infty}(X)$ and $\epsilon>0$ be given. Choose $x_0\in X$ and  $R\geq 0$ such that 
$\lvert f(y)\rvert\leq\frac{\epsilon}{2}$ for $y\in X\setminus B(x_0,R)$. Furthermore, let  $M>0$  be such that $M>2\lVert f\rVert_{\infty}$.
Given $t>0$, by  Lemma \ref{expdecay}, there is a constant $r_t$ such that
\begin{align*}
\int_{X\setminus B(x,r_t)}h_t(x,y)\,dy\leq \frac{\epsilon}{M},\quad \forall x\in X.
\end{align*}
Now for $x\in X$ such that $d(x,x_0)>R+r_t$ we have $B(x,r_t)\subset X\setminus B(x_0,R)$ which   together with the positivity of $h_t$ yields that 
\begin{align*}
 \Bigl\lvert \int_X h_t(x,y)f(x)\,dy\Bigr\rvert  &\leq \int_{X\setminus B(x_0,r_t)}h_t(x,y)\lvert f(y) \rvert\, \,dy\\
&+\int_{B(x,r_t)}h_t(x,y)\lvert f(y)\rvert\,dy\\
&\leq \frac{M}{2}\frac{\epsilon}{M}+\frac{\epsilon}{2}\int_{B(x,r_t)}h_t(x,y)\,dy\\
&\leq \frac{\epsilon}{2}+\frac{\epsilon}{2}=\epsilon.
\end{align*}
Thus, $e^{t\Delta}f\in C_0^{\infty}(X)$.
\item[Conservative]
This we easily obtain by the stochastic completeness of $X$, Lemma \ref{complete}, 
\begin{align*}
e^{t\Delta}1_X=\int_{X}h_t(x,y)\,dy=1_X.
\end{align*}
\item[Strong continuity]
First, we observe that, since $f\in C_0^{\infty}(X)$ has a bounded differential, it is uniformly continuous. Let $\epsilon>0$ be given, let $\delta>0$ be such that $\lvert f(x)-f(y)\rvert <\frac{\epsilon}{2}$ for $d(x,y)<\delta$ 
and let $M>4\lVert f\rVert_{\infty}$. Then it follows from Lemma \ref{expdecay} that there is a $t_0>0$ such that for $0<t<t_0$ we have:
\begin{align*}
\int_{X\setminus B(x,\delta)}h_t(x,y)\,dy\leq \frac{\epsilon}{M}.
\end{align*}
Hence, we obtain for $x\in X$ and $0<t<t_0$ 
\begin{align*}
\lvert (e^{t\Delta}f)(x)-f(x)\rvert&=\Bigl\lvert \int_{X}h_t(x,y)(f(y)-f(x))\,dy\Bigr\rvert\\
&\leq \int_{B(x,\delta)}h_t(x,y)\lvert f(x)-f(y)\rvert\,dy\\
&+ \int_{X\setminus B(x,\delta)}h_t(x,y)\lvert f(x)-f(y)\rvert \,dy\\
&\leq\frac{\epsilon}{2}\int_{B(x,\delta)}h_t(x,y)\,dy+2\lVert f\rVert_{\infty}\frac{\epsilon}{M}\\
&\leq \frac{\epsilon}{2}+\frac{\epsilon}{2}=\epsilon.
\end{align*}
Hence, $\lVert e^{t\Delta}f-f\rVert_{\infty}\leq \epsilon$ for $0<t<t_0.$
\item[Continuous paths]
Let $\epsilon>0$ and choose $t>0$ such that $\epsilon=Dht$. Then we get:
\begin{align}\label{expdecay2}
\frac{1}{t}\int_{X\setminus B(x,\epsilon)}h_t(x,y)\,dy\leq \frac{\kappa e^{-\eta \epsilon}}{t}.
\end{align}
Hence, (\ref{expdecay2}) converges to zero as $t\to 0$. Now Theorem \ref{thm:conpath} implies
the assertion.
\end{description}
\end{proof}

With the formula for the Green kernel on non-compact  simply connected harmonic manifolds, 
we obtain: 
 \begin{lemma}\label{fenite}
 The Green kernel is finite for all $x,y\in X$ with $x\neq y$.
 \end{lemma}
 \begin{proof}
 Since $X$ has purely exponential Volume growth we have:
 $$G(x,y)=\frac{1}{\vol(S_yX)}\int_{d(x,y)}^{\infty}\frac{1}{A(r)}\,dr\leq\int_{0}^{\infty}ce^{-2\rho r}\,dr<\infty$$
 \end{proof}
 
\begin{bem}
Note that Lemma \ref{fenite} holds in the more general case of a non-flat non-compact harmonic manifold.

\end{bem} 

 \begin{satz}[\cite{Grigor99}]
 The following are equivalent.
 \begin{enumerate}
 \item The Brownian Motion on $X$ is transient meaning, that for every point $x\in X$ we have: 
 \begin{align*}
 \mathbb{P}_x\Bigl(\lim_{t\to\infty} d(x, B_t(x)))\to\infty \Bigr)=1.
 \end{align*} 
 \item For all $x\neq y$ the Green kernel is finite. 
 \end{enumerate}
 \end{satz}
 Now with  Lemma \ref{fenite} and the theorem above, it immediately follows:
\begin{satz}\label{thm:trans}
 Let $X$ be a simply connected non-compact harmonic manifold of rank one and $(B_t)_{t\geq0}$ the Brownian motion on $X$. Then 
 for any point $x\in X$,
 $$\mathbb{P}_x\Bigl(\lim_{t\to\infty} d(x, B_t(x)))\to\infty \Bigr)=1.$$ 
Hence  $B_t$ is transient. 
\end{satz}

\section{Eigenfunctions via Exit Time of Brownian Motion}

\subsection{Fundamentals on the heat kernel and heat semi-group, and their connection to the Brownian motion}
In this subsection, we collect some well-known properties of the heat kernel and the heat semi-group which are of use in later stages of this section.
\begin{prop}\label{Delta}
Let $\psi\in C^{\infty}_c(X)$. Then
$$ \lim_{t\to 0}\frac{e^{t\Delta}\psi-\psi}{t}=\Delta \psi \text{ in } \Bigl ( B(X),\lVert\cdot\rVert_{\infty}\Bigr).$$
\end{prop}
\begin{proof}
Let $f_1\in C^{\infty}(M)$, $f_2\in C^{\infty}_c(X)$, $G\subset X$ a pre-compact domain with smooth boundary and $\operatorname{supp}(f_2)\subset G$. Then by Green's identity, we have:
\begin{align}\label{eq:Deltaselfadjoint}
\int_X\Delta f_1(x)\cdot f_2(x)\,dx=\int_X \Delta f_2(x)\cdot f_1(x)\,dx.
\end{align}
For given $f_2\in C^{\infty}_c(X)$ let $u(t,x):=(e^{t\Delta}f_2(x))$ for $t\geq 0 $. 
Then for all $t>0$
\begin{align*}
u(t,x)=\int_Xh_t(x,y)f_2(y)\,dx.
\end{align*}
Hence, by using (\ref{eq:Deltaselfadjoint}) and the heat equation we have :
\begin{align*}
u'(t,x)&=\int_X \left(\frac{\partial }{\partial t} h_t(x,y)\right) f_2(y)\,dy\\
&=\int_X\Delta_y h_t(x,y)f_2(y)\,dy\\
&=\int_X h_t(x,y)\Delta f_2(y)\,dy.
\end{align*}
Therefore $u'(t,x)\to \Delta f_2(x)$ as $t\to 0$. Hence, $u(\cdot,x)$ is differentiable as a function on $[0,\infty)$ for every $x\in M$. Therefore:
\begin{align*}
e^{t\Delta}f_2(x)-f_2(x)&=u(t,x)-u(0,x)\\
&=\int_0^t u'(s,x)\,ds\\
&=e^{s\Delta}\Delta f_2(x)\,ds.
\end{align*}
Hence, for $t>0$ we have:
\begin{align*}
\left\lVert\frac{e^{t\Delta}f_2-f_2}{t}-\Delta f_2\right\rVert_{\infty}&\leq \frac{1}{t}\int_0^t\lVert e^{s\Delta}\Delta f_2-\Delta f_2\rVert_{\infty}\,ds\\
&\leq \sup_{0\leq s\leq t}\lVert e^{s\Delta}\Delta f_2-\Delta f_2\rVert_{\infty},
\end{align*}
which by the strong continuity of the heat semi-group converges to $0$ as $t\to 0$. 
\end{proof}
\begin{folg}\label{folg:Delta}
For $f\in B(X)$ with compact support we have
$$\frac{e^{t\Delta}f-f}{t}\to \Delta f\text{ as }t\to 0$$
in the sense of distributions on $X$. 
\end{folg}
\begin{proof}
Denote by $\mathcal{D}(X)$ the space of distributions ( i.e. the dual of $C^{\infty}_c(X)$, equipped with the weak* topology) on $X$ and by $(\cdot,\cdot)$ the pairing between $\mathcal{D}(X)$ and $C^{\infty}_c(X)$. Then, for $f\in B(X)$ with compact support and for $\psi\in C^{\infty}_c(X)$ the distribution $\Delta f$ is defined as $(\Delta f,\psi)=(f,\Delta \psi)$. Using the self-adjointness of $e^{t\Delta}$ on $L^2(M)$ we get:
\begin{align*}
\Bigl(\frac{e^{t\Delta}-\operatorname{id}}{t}f,\psi\Bigr)=\Bigl(f,\frac{e^{t\Delta}-\operatorname{id}}{t}\psi\Bigr).
\end{align*}
Since $\frac{e^{t\Delta}-\operatorname{id}}{t}\psi\to \Delta \psi$ uniform as $t\to 0$ by Proposition \ref{Delta} and since $f\in L^1(X)$, we conclude
\begin{align*}
\Bigl(f,\frac{e^{t\Delta}-\operatorname{id}}{t}\psi\Bigr)\to (f,\Delta \psi)=(\Delta f,\psi), 
\end{align*}
as $t\to0$. 
\end{proof}

Let $G\subset X$ be a pre-compact domain with smooth boundary $\partial G$. 
 Using an idea described in \cite[Chapter VII]{chavel1984eigenvalues} 
  we can obtain the Dirichlet heat kernel $h_G$ from the heat kernel $h_t$ in the following way:
 For each $x\in \overline{G}$ there is a
  continues function $g(\cdot,x,\cdot):[0,\infty)\times \overline{G} \to \R$ such that $g$ is smooth on $(0,\infty)$ and satisfies:
 \begin{align*}
 \frac{\partial g}{\partial t}&=\Delta_y g\text{ on } (0,\infty)\times \overline{G},\\
 g(0,x,y)&=0\quad\forall y \in\overline{G},\\
 g(t,x,y)&=-h_t(x,y)\quad \forall t>0 \text{ and }\forall y\in\partial G.
 \end{align*}
 Then we get for all $t>0$ and $x,y\in \overline{G}$
 \begin{align}\label{eq:dirichletg}
 h_G(t,x,y)=h_t(x,y)+g(t,x,y).
 \end{align}
 Note that since by definition $h_G$ vanishes at the boundary of $G$ one can use equation (\ref{eq:dirichletg}) to define $g$.
 We also obtain the parabolic maximum principle for the heat equation (see \cite[SectionVIII.1]{chavel1984eigenvalues}):
 Define for $T>0$ 
 \begin{align*}
 G_T&:=(0,T)\times G\\
 \text{and }\partial_PG_T&:=(\{0\}\times\overline{G}\cup ([0,T]\times\partial G),
 \end{align*}
 the latter is called the parabolic boundary of $G_T$.
 \begin{prop}[{\cite[Chapter VII]{chavel1984eigenvalues}}]\label{prop:parabolic}
 Let $u:\overline{G_T}\to \R$ be a continuous function which is a $C^{\infty}$-solution of the heat equation on $G_T$. Then:
 \begin{align*}
 \sup_{[0,T]\times \overline{G}} u&=\sup_{\partial_PG_T} u\\
 \text{and } 
  \inf_{[0,T]\times \overline{G}} u&=\inf_{\partial_PG_T} u.
  \end{align*}
  \end{prop}
 \begin{folg}\label{folg:dirichlet1}
 Let $g$ be as in equation $(\ref{eq:dirichletg})$ then
 \begin{align*}
 g(t,x,y)\leq 0\quad \forall t> 0,x,y\in \overline{G}.
 \end{align*}
 Therefore, 
 \begin{align*}
 h_G(t,x,y)\leq h_t(x,y)\quad\forall t>0,x,y\in\overline{G}
 \end{align*} 
 and consequently
 \begin{align*}
 \int_G h_G(t,x,y)\,dy\leq 1.
 \end{align*}
 \end{folg}
 With this, we can  obtain bounds on the function $g$.
 \begin{lemma}\label{lemma:boundg}
 There are constants $\alpha,\beta>0$ such that for every compact subset $K$ of $G$ we have
 \begin{align*}
 \sup_{y\in \overline{G}}(-g(t,x,y))\leq \alpha e^{-\delta^2\beta/t}\quad \forall t\in(0,1),x\in K,
 \end{align*}
 where $\delta$ is the distance from $K$ to $\partial G$. 
 \end{lemma}
 \begin{proof}
 Given $x\in K$ and $t\in (0,1)$, since $g(0,x,y)=0$ for all $y\in \overline{G}$ the parabolic maximum principle (Proposition \ref{prop:parabolic}) gives:
 \begin{align*}
 \sup_{y\in\overline{G}}(-g(t,x,y))&=\sup_{(s,y)\in \partial_PG_t}(-g(s,x,y))\\
 &\leq \sup_{0<s\leq t,y\in\partial G}h_s(x,y).
 \end{align*}
Note that $\delta\leq d(x,y)$ for all $x\in K$ and $y\in \partial G$. Hence, we get the assertion by choosing $\alpha>0$ large and $\beta>0$ small enough in Lemma \ref{lemma:heateastimate4}. 
 \end{proof}

Let $(B_t)_{t\geq 0}$ be the Brownian motion on $X$. Let $C_X$  denote the space of continuous paths $\gamma:[0,\infty)\to X$  on $X$ and  
\begin{align*}
\pi_t:C_X&\to X\\
\gamma&\mapsto \gamma(t)
\end{align*}
the evaluation map.
Equip $C_X$  with the the sigma-algebra $\mathcal{C}$ generated by $\pi_t$
and denote for  $x\in X$ by $\mathbb{P}_x$ the transition probability of $(B_t)_{t\geq 0}$ with initial distribution $\delta_x$. Note that a path that dose not start in $x$ has zero probability.
 Given a pre-compact domain $G\subset X$, the exit time of the Brownian motion  $(B_t)_{t\geq 0}$ from $G$
 \begin{align*}
 \tau_G:C_X\to[0,\infty]
 \end{align*}
  is defined by 
\begin{align*}
\tau_{G}(\gamma)=\inf\{t>0:\gamma(t)\in X\setminus G\}.
\end{align*}
Note that, since the Brownian motion is transitive, $\mathbb{P}_x(\tau_G< \infty)=1$ for all pre-compact domains $G$ in $X$ and $x\in G$. Hence, the exit point of the Brownian motion from $G$ 
\begin{align*}
\pi_G:\{\gamma\in C_X\mid \tau_G<\infty\}&\to X\\
\gamma&\mapsto \gamma(\tau_G(\gamma))
\end{align*}
is defined $\mathbb{P}_x$ almost everywhere. 
From now on if not mentioned otherwise assume $x\in G$. 
Since the sample path is continuous, we have $\mathbb{P}_x(\pi_G\in\partial G)=1$. 
Furthermore, note that if 
\begin{align*}
\theta_t:C_X&\to C_X\\
(s\mapsto \gamma(s))&\mapsto (s\mapsto \gamma(s+t))
\end{align*}
is the time shift, then $\tau_G\circ\theta_t=\tau_G-t$ for all $\gamma\in C_x$ with $t\leq \tau_G(\gamma)<\infty$. 
It is well known (see for instance \cite{Grigor99}) that for every $x\in G$ 
\begin{align}\label{eq:dirichletwsk}
\mathbb{P}_x(\tau_G\geq t)=\int_Gh_G(t,x,y)\,dy.
\end{align} 
To simplify the notion we will shorten $\tau_G=\tau$ and $\pi_G=\pi$, omit the argument of functions unless it is fixed, and denote the probability density function associated to a cumulative distribution function $F$ by $dF$. 

\begin{prop}\label{prop:measurheat}
For any $\lambda\in \C$ with $\operatorname{Re}\lambda
> \lambda_1$ there is a constant $C_{\lambda}$ such that:
\begin{align*}
\int_{C_X} \lvert e^{-\lambda\tau}\rvert\,d\mathbb{P}_x\leq C_{\lambda}
\end{align*}
for all $x\in G$. In other words, the  measure $e^{-\lambda\tau}\,d\mathbb{P}_x$ has finite total variation for all $x\in G$. 
\end{prop}
\begin{proof}
$\mathbb{P}_x$ is a probability measure 
and $\lvert e^{-\lambda s}\rvert=e^{-(\operatorname{Re}\lambda) s }$. Hence, for $\operatorname{Re}
 \lambda\geq 0$ we are done. 
Therefore we may without loss of generality assume that $\lambda$ is real and $\lambda_1<\lambda<0$. Define $F(s)=\mathbb{P}_x(\tau\geq s)$.
Then $F$ is a monotonically decreasing function. From (\ref{eq:dirichletwsk}), Corollary \ref{folg:dirichlet1} together with the second formula from Lemma \ref{lemma:heateastimate4} we obtain for $t>1$
\begin{align*} 
\int_G h_G(t,x,y)\,dy&\leq \int_G h_t(x,y)\,dy\\
&\leq K\Bigl(1+\frac{\operatorname{diam} G}{t}\Bigr)\cdot e^{\lambda_1 t}\\
&\leq C_1e^{\lambda_1 t}
\end{align*}
for some constant $C_1>0$, where $\operatorname{diam} G$ is the diameter of $G$. 
Therefore $F(s)\leq C_1e^{\lambda_1 s}$ for $s>1$ and hence $e^{-\lambda s}F(s)\to 0$ as $s\to\infty$.
Hence, via integration by parts we obtain:
\begin{align*}
\int_{C_X} e^{-\lambda \tau}\,d\mathbb{P}_x&=-\int_0^{\infty}e^{-\lambda s}dF(s)\\
&=-[e^{-\lambda s} F(s)]_0^{\infty} +\int_0^{\infty} \lambda e^{-\lambda s}F(s)\,ds\\
&\leq 1+\int_0^{1}(-\lambda)e^{-\lambda s}\,ds+C_1\int_1^{\infty} e^{(\lambda_1-\lambda)s}\,ds\\
&=C_{\lambda}.
\end{align*}
\end{proof}
\begin{lemma}\label{lemma:heatball}
There are constants $\gamma,\beta'>0$, such that for any compact subset $K$ of $G$, for all $x\in K$ and $0<t<\min(1,\delta/(D\rho))$, where $D$ as in Lemma \ref{lemma:heateastimate4}, $\rho$ is half the mean curvature of the horosphere and $\delta=\operatorname{dist}(K,\partial G)$. We have:
\begin{align*}
\mathbb{P}_x(\tau<t)\leq \gamma e^{-\gamma^2\beta'/t}.
\end{align*}
\end{lemma}
\begin{proof}
By using equation (\ref{eq:dirichletg}), Lemma \ref{lemma:boundg}, equation (\ref{eq:dirichletwsk}) and Lemma \ref{expdecay} we get: 
\begin{align*} 
\mathbb{P}_x(\tau<t)&=1-\int_G h_G(t,x,y)\,dy\\
&=1-\int_G h_t(x,y)\,dy-\int_G g(t,x,y)\,dy\\
&\leq \int_{X\setminus B(x,\delta)} h_t(x,y)\,dy+\operatorname{vol}(G)\alpha e^{-\delta\beta/t}\\
&\leq \kappa e^{-\delta^2\eta/t} +\operatorname{vol}(G)\alpha e^{-\delta\beta/t}\\
&\leq \gamma e^{-\gamma^2\beta'/t},
\end{align*}
with $\beta'=\min(\eta,\beta)$ and $\gamma=\kappa +\alpha\operatorname{vol}(G)$.
\end{proof}
Before starting with the proof of Theorem \ref{thm:exitbrown} we need some preliminary definitions:
We fix $\lambda\in \C$ with $\operatorname{Re}\lambda>\lambda_1$ and define the measure $\nu_{x,\lambda}$ on $\partial G$ by:
\begin{align*}
\nu_{x,\lambda}=\pi_{*}(e^{-\lambda\tau}d\mathbb{P}_x).
\end{align*}
Since $\mathbb{P}_x(\pi\in\partial G)=1$ for $x\in G$, we have that $\nu_{x,\lambda}$ has support on $\partial G$ and by Proposition \ref{prop:measurheat} is a complex 
measure 
of finite variation. Furthermore, we fix a non-constant continuous function $\phi:\partial G\to \C$ and define 
\begin{align*}
\psi(x):=\int_{\partial G} \phi(y)\,d \nu_{x,\lambda}.
\end{align*}
We can assume  $\phi$ to be continuous with compact support in $X$. Otherwise, we may extend $\phi$ in a suitable way. 
Define 
\begin{align*}
\Phi:\{\gamma\in C_X\mid \tau(\gamma)<\infty\}&\to \C\\
\gamma&\mapsto \phi(\pi(\gamma))
\end{align*}
and interpret $\Phi$ as the random variable $\phi(B_{\tau})$. Hence:
\begin{align*}
\psi(x)&=\int_{C_X}e^{-\lambda\tau}\Phi\,d\mathbb{P}_x\\
&=\int_{C_X}e^{-\lambda\tau}\phi(B_{\tau})\,d\mathbb{P}_x\\
&=\mathbb{E}_x(e^{-\lambda \tau}\phi (B_{\tau})).
\end{align*}
Note that for $x\in X\setminus \overline{G}$ we have  $\mathbb{P}_x(\tau=0)=1$ and $\mathbb{P}_x(\pi=x)=1$ hence $\psi(x)=\phi(x)$ for $x\in X\setminus\overline{G}$. Therefore, $\psi$ is a bounded function on $X$ with compact support. 
Lastly we define for $t>0,$ 
\begin{align}\label{no:pathspace}
\{t\leq \tau<\infty\}
\end{align}
 to be the subset of $C_X$ with exit times in the range $[t,\infty)$. Furthermore we define $\{\tau<t\}$ accordingly. 
 \subsection{Construction of Eigenfunctions}
 \begin{satz}\label{thm:exitbrown}
 Let $X$ be a non-compact simply connected harmonic manifold of rank one and let $\{B_t\}_{t\geq0}$ be the Brownian motion in $X$. Let $\lambda_1 < 0$ be the supremum of the Laplacian $\Delta$ on $L^2(X)$. Let $G\subset X$ be a pre-compact domain in $X$ with smooth boundary $\partial G$ and let $\tau$ be the first exit time from $G$ of $\{B_t\}_{t\geq0}$. Then for any $\lambda\in \C$ with 
 $\operatorname{Re}\lambda >\lambda_1$ and for any non constant continuous function $\phi:\partial G\to\C$ the function 
 $$\psi(x):=\mathbb{E}_x(e^{-\lambda \tau}\phi (B_{\tau}))$$
 is analytic on $G$ and is an eigenfunction of $\Delta$ on $G$ with eigenvalue $\lambda$ and boundary value $\phi$, meaning 
 \begin{align*}
 \Delta\psi(x)&=\lambda\psi\quad\forall x\in G\\
 \text{and}\\
 \psi(x)&\to \phi(y)\text { for } G\ni x\to y\in \partial G.
 \end{align*}
 \end{satz}
The proof will be conducted via the following series of lemmas and propositions and as mentioned in the introduction we can follow along the lines of the proof in \cite{biswas2019sullivans}. 
 
\begin{lemma}\label{lemma:semiheatkompact}
For any compact subset $K$ of $G$ we have
\begin{align*}
\frac{e^{t\Delta}\psi-\psi}{t}\to\lambda\psi
\end{align*}
uniform on $K$. 
\end{lemma}
\begin{proof}
We have for $t>0$ that the exit point of a path does not change by translation by $t$ on $\{t\leq\tau<\infty\}$ and therefore $\pi\circ\theta_t=\pi$ on  $\{t\leq\tau<\infty\}$. Hence $\Phi\circ\theta_t=\Phi$ on $\{t\leq\tau<\infty\}$.

Furthermore, we have $\tau\circ\theta=\tau-t$ on $\{t\leq\tau<\infty\}$. By the semi-group property of the heat semi-group, we can represent $\theta_{t*}\mathbb{P}_x$ as a convex combination of the measures $\mathbb{P}_y$ by:
\begin{align}\label{eq:convemeasure}
\theta_{t*}(\mathbb{P})_x(B)=\int_X \mathbb{P}_{y}(B) h_t(x,y)\,dy.
\end{align}
Where $B\subset C_X$ is a Borel set. 
With this, we obtain for $x\in K$:
\begin{align*}
(e^{t\Delta}\psi)(x)&=\int_X h_t(x,y)\psi(y)\,dy\\
&=\int_X h_t(x,y)\Bigl (\int_{C_x}e^{-\lambda\tau}\Phi\,d\mathbb{P}_y\Bigr )\,dy\\
&=\int_{C_X} e^{-\lambda\tau}\Phi\,d(\theta_{t*}(\mathbb{P}_x))\\
&=\int_{\{t\leq \tau<\infty\}}e^{-\lambda\tau\circ\theta_t}\Phi\circ\theta_t\,d\mathbb{P}_x+\int_{\{t<\tau\}}e^{-\lambda\tau\circ\theta_t}\Phi\circ\theta_t\,d\mathbb{P}_x\\
&=\int_{C_X}e^{-\lambda(\tau-t)}\Phi\,d\mathbb{P}_x-\int_{\{\tau<t\}}e^{-\lambda(\tau-t)}\Phi\,d\mathbb{P}_x\\
&\quad +\int_{\{\tau<t\}}e^{-\lambda \tau\circ\theta_t}\Phi\circ\theta_t\,d\mathbb{P}_x\\
&=e^{\lambda t}\psi(x)-A(x,t)+B(x,t),
\end{align*}
where $A(x,t)=\int_{\{\tau<t\}}e^{-\lambda(\tau-t)}\Phi\,d\mathbb{P}_x$
and $B(x,t)=\int_{\{\tau<t\}}e^{-\lambda \tau\circ\theta_t}\Phi\circ\theta_t\,d\mathbb{P}_x$.
Therefore
\begin{align}
\frac{(e^{t\Delta}\psi)(x)-\psi(x)}{t}=\frac{e^{\lambda t}-1}{t}\psi(x)-\frac{1}{t}A(x,t)+\frac{1}{t}B(x,t).
\end{align}
 It is sufficient to show that $A(x,t)=o(t)$ and $B(x,t)=o(t)$ as $t\to0$ uniformly in $x\in K$.
Denote by $\delta$ the distance from $K$ to $\partial G$ and let $M>0$ be such that $\lvert \phi\rvert <M$ and therefore $\lvert\Phi\rvert< M$ on $\{t<\infty\}=\{\omega\in C_X\mid \tau(\omega)<\infty\}$.
Then by using the constants from Lemma \ref{expdecay}, $0<t<\min\{1,\delta/(D2\rho)\}$ and Lemma \ref{lemma:heatball} we get for $x\in K$ 
\begin{align*}
\frac{1}{t}\lvert A(x,t)\rvert&\leq \frac{1}{t}\int_{\{\tau< t\}}\lvert e^{-\lambda \tau} e^{\lambda t}\rvert \,d\mathbb{P}_x\\
&\leq \frac{1}{t} e^{2\lvert \lambda\rvert t} M\mathbb{P}_x(\tau<t)\\
&\leq \frac{1}{t}e^{2\lvert \lambda\rvert t} M \gamma e^{-\delta^2\beta'/t}\\
&\to 0 \text{ uniformly in } x\in X \text{ as } t\to 0,
\end{align*} 
where $\gamma$ and $\beta'$ are the constants from Lemma \ref{lemma:heatball}.
Finding an estimate for $B(x,t)$ is more involved. We have $\operatorname{Re}
 \lambda >\lambda_1$ hence we can choose $p>1$ such that $\lambda':=p\lambda$ also satisfies $\operatorname{Re} 
 \lambda'>\lambda_1$. We now want to estimate the $L^p$ norm of $e^{-\lambda\tau\theta_t}$ with respect to $\mathbb{P}_x$ and apply H\"older's inequality.
By using equation (\ref{eq:convemeasure}), Proposition \ref{prop:measurheat} and $\mathbb{P}_y(\tau=0)=1$ for $y\in X\setminus \overline{G}$ we obtain:
\begin{align*}
\int_{C_X}\lvert e^{-\lambda \tau\circ\theta_t}\rvert^p=&\int_X\Bigl ( \int_{C_X}\lvert e^{-\lambda'\tau}\rvert \,d\mathbb{P}_y\Bigr)h_t(x,y)\,dy\\
=&\int_{D}\Bigl ( \int_{C_X}\lvert e^{-\lambda'\tau}\rvert \,d\mathbb{P}_y\Bigr)h_t(x,y)\,dy\\&+\int_{X\setminus\overline{D}}\Bigl ( \int_{C_X}\lvert e^{-\lambda'\tau}\rvert \,d\mathbb{P}_y\Bigr)h_t(x,y)\,dy\\
=&\int_{D}\Bigl ( \int_{C_X}\lvert e^{-\lambda'\tau}\rvert \,d\mathbb{P}_y\Bigr)h_t(x,y)\,dy\\&+\int_{X\setminus\overline{D}}h_t(x,y)\,dy\\
\leq&\int_D C_{\lambda'}h_t(x,y)\,dy+1.
\end{align*}
By choosing $q$ H\"older conjugated to $p$ and using H\"older's inequality and Lemma \ref{lemma:heatball} we conclude:
\begin{align*}
\frac{1}{t}\lvert B(x,t)\rvert & \leq \frac{1}{t}\int_{ \{\tau<t \}} \lvert e^{-\lambda \tau\circ\theta_t}\rvert \cdot \lvert \Phi\circ\theta_t\rvert\,d\mathbb{P}_x\\
&\leq \frac{1}{t} M \int_{C_X} \lvert e^{-\lambda \tau\circ\theta_t}\rvert \cdot\mathcal{X}_{ \{\tau<t\} }\,d\mathbb{P}_x\\
&\leq \frac{M}{t}\Bigl (\int_{C_X}\lvert e^{-\lambda\tau\circ\theta_t}\rvert^p\,d\mathbb{P}_x\Bigr )^{1/p}\mathbb{P}_x(\tau<t)^{1/q}\\
&\leq \frac{M}{t}(C_{\lambda'}+1)^{1/p}\gamma e^{-\gamma^2\beta'/(qt)}\\
&\to 0 \text{ uniformly in } x\in X \text{ as } t\to 0,
\end{align*}
where for a Borel set $E$, $\mathcal{X}_E$ denotes the characteristic function. 
 This concludes the proof. 
\end{proof}
\begin{prop}\label{prop:anaexit}
The function $\psi$ is analytic on $G$ and satisfies
\begin{align*}
\Delta\psi=\lambda\psi
\end{align*}
on $G$.
\end{prop}
\begin{proof}
Let $(\cdot,\cdot )$ denote the pairing between distributions on $G$ and $C^{\infty}_c(G)$. Given $\varphi\in C^{\infty}_c(G)$ with $K=\operatorname{supp}(\varphi)\subset G$, then on one hand  it follows from Lemma \ref{lemma:semiheatkompact} that for $t\to 0$:
\begin{align*}
\Bigl(\frac{e^{t\Delta}\psi-\psi}{t},\varphi\Bigr)\to (\lambda \psi,\varphi)
\end{align*}
on the other hand we have by Corollary \ref{folg:Delta} for $t\to 0$
\begin{align*}
\Bigl(\frac{e^{t\Delta}\psi-\psi}{t},\varphi \Bigr)\to (\Delta \psi,\varphi).
\end{align*}
It follows $\Delta\psi=\lambda \psi$ as distributions on $G$ and hence by elliptic regularity and since $X$ is Einstein $\psi$ is analytic on $G$ and $\Delta\psi=\lambda\psi$ as functions on $G$.
\end{proof}
This proves the first part of Theorem \ref{thm:exitbrown}. 
For the second part, we need two more lemmas. 
\begin{lemma}\label{lemma:thmexit1}
Let $z\in \partial G$. Then for any $t>0$:
\begin{enumerate}
\item $\mathbb{P}_x(\tau\geq t)\to 0$ as $G\ni x\to z\in\partial G$.
\item $\int_{\{\tau\geq t\}}\lvert e^{\lambda \tau}\rvert\,d\mathbb{P}_x\to 0$ as $G\ni x\to z\in\partial G$.
\end{enumerate}
\end{lemma}
\begin{proof}
For a fixed $t>0$ the function  $h_G(t,\cdot,\cdot)$ is a continuous on $\overline{G}\times\overline{G}$ and vanishes for one of the values in $\partial G$. Hence, $h_G(t,x,y)\to 0$ uniformly in $y$ as $x\to z\in\partial G$. Therefore we have  for $G\ni x\to z\in\partial G$
\begin{align*}
\mathbb{P}_x(\tau\geq t)=\int_G h_G(t,x,y)\,dy\to 0.
\end{align*}
This proves the first assertion.
For the second we can assume without loss of generality $\lambda$ to be real and
 $\lambda_1<\lambda<0$.
For $x\in G$ and $s>0$ define the function 
\begin{align*}
F_x(s):=\mathbb{P}_x(\tau\geq s)=\int_G h_G(s,x,y)\,dy.
\end{align*}
Then, as in the proof of Proposition \ref{prop:measurheat} we have $F_x(s)\leq C_1e^{\lambda_1 s}$ for some constant $C_1>0$ independent of $x$. Now we can employ integration by parts:
\begin{align*}
\int_{\{\tau\geq t\}} e^{-\lambda \tau}\,d\mathbb{P}_x&=-\int_t^{\infty} e^{-\lambda s}\,dF_x(s)\\
&=e^{-\lambda t}\mathbb{P}_x(\tau\geq t)+(-\lambda)\int_t^{\infty}F_x(s)\,ds\\
&=e^{-\lambda t}\mathbb{P}_x(\tau\geq t)+(-\lambda)\int_G\int_t^{\infty}e^{-\lambda s}h_G(s,x,y)\,ds\,dy.
\end{align*}
Now $e^{-\lambda t}\mathbb{P}_x(\tau\geq t)\to 0$ as $x\to z\in \partial G$ by the first assertion. 

To obtain an estimate on the second term we use the large-scale estimate from Lemma \ref{lemma:heateastimate4} to find a constant $C_2>0$ such that $h(s,x,y)\leq C_2e^{\lambda_1 s}$. This yields, since $\lambda>\lambda_1$ and by the second part of Corollary \ref{folg:dirichlet1}
\begin{align*}
\int_G\int_t^{\infty}e^{-\lambda s}h_G(s,x,y)\,ds\,dy&\leq \int_G\int_t^{\infty}e^{-\lambda s}h_s(x,y)\,ds\,dy\\
&\leq C_2 \int_G\int_t^{\infty}e^{-\lambda s}e^{\lambda_1 s} \,ds\,dy\\
&<\infty.
\end{align*}
Hence, dominant convergence applies and we get for $x\to z\in\partial G$:
\begin{align*}
\int_G\int_t^{\infty}e^{-\lambda s}h_G(s,x,y)\,ds\,dy\to 0.
\end{align*}
\end{proof}
For $x\in X$ and $\delta>0$ we define the exit time $\tau_{x,\delta}:C_X\to [0,\infty]$ of the Brownian motion from the open ball $B(x,\delta)$ by $\tau_{x,\delta}(\gamma):=\inf\{t>0\mid \gamma(t)\in X\setminus B(x,\delta)\}$.
\begin{lemma}\label{lemma:thmexit2}
For any $\delta>0$
\begin{align*}
\mathbb{P}_x(\tau_{x,\delta}<t)\to 0
\end{align*}
as $t\to 0$ uniformly in $x$ on compact sets. 
\end{lemma}
\begin{proof}
For given $x\in X$ and $\delta>0$. Let $B=B(x,\delta)$ denote the open ball of radius $\delta$ around $x$ and $h_B$ the Dirichlet heat kernel of $B$. Let $g_B$ be the corresponding function from equation (\ref{eq:dirichletg}). Hence Lemma \ref{lemma:boundg} applies to $K=\{x\}$ and we have constants $\alpha,\beta$ independent of $x$ such that 
\begin{align*}
 \sup_{y\in \overline{B}}(-g_B(t,x,y))\leq \alpha e^{-\delta^2\beta/t}.
 \end{align*}
 With Lemma \ref{expdecay} this yields for $0<t<\delta/(D2\rho):$ 
 \begin{align*}
 \mathbb{P}_x(\tau_{x,\delta}<t)&=1-\int_B h_B(t,x,y)\,dy\\
&=1-\int_B h_t(x,y)\,dy-\int_B g_B(t,x,y)\,dy\\
&=\int_{X\setminus B}h_t(x,y)\,dy-\int_B g_B(t,x,y)\,dy\\
&\leq \kappa e^{-\delta\eta/t}+\vol(B)\cdot\alpha e^{-\delta^2\beta/t},
\end{align*}
where we again use the constants obtained in the previous lemmas. 
Since $X$ is harmonic, the volume of $B$ only depends on $\delta$. Hence $\kappa e^{\delta\eta/t}+\vol(B)\cdot\alpha e^{-\delta^2\beta/t}\to 0$ as $t\to 0$ uniformly in $x$. 
\end{proof}
\begin{prop}\label{prop:exitcont}
Let $\phi:\partial G\to \C$ be a continuous function, $\operatorname{Re}\lambda>\lambda_1$ 
 and $\psi(x):=\mathbb{E}_x(e^{-\lambda\tau}\phi(B_{\tau}))$.
 Then $\psi(x)\to\phi(z)$ as $G\ni x\to z\in \partial G$.
 \end{prop}
 \begin{proof}
 Let $z\in \partial G$ and fix $\epsilon>0$. We choose $\delta>0$ such that 
 \begin{align*}
 \lvert \phi(y)-\phi(z)\rvert\leq \epsilon\quad \forall y\in \partial G\cap \overline{B(z,\delta)}.
 \end{align*}
 Furthermore, we fix a constant $M>0$ such that $M>\lVert \phi\rVert_{\infty}$. 
 Let $t>0$ and $\tau_{z,\delta}$ be the exit time from the open ball $B(z,\delta)$. Then with notation analogous to the previous notation (see (\ref{no:pathspace})):
 \begin{align*}
 C_X=\{\tau<t,\tau<\tau_{z,\delta}\} \dot{\cup} \{\tau <t,\tau\geq \tau_{z,\delta}\}\dot{\cup}\{\tau\geq t\}.
 \end{align*}
 Hence we obtain:
 \begin{align*}
 \lvert \psi(x)-\phi(z)\rvert&\leq\int_{C_X}\lvert e^{-\lambda\tau}\phi(B_{\tau})-\phi(z)\rvert\,d\mathbb{P}_x\\
& \leq C(x,t)+D(x,t)+E(x,t),
 \end{align*}
 where:
 \begin{align*}
 C(x,t)&=\int_{\{\tau<t,\tau<\tau_{z,\delta}\} }\lvert e^{-\lambda\tau}\phi(B_{\tau})-\phi(z)\rvert\,d\mathbb{P}_x,\\
 D(x,t)&=\int_{ \{\tau <t,\tau\geq \tau_{z,\delta}\} }\lvert e^{-\lambda\tau}\phi(B_{\tau})-\phi(z)\rvert\,d\mathbb{P}_x\\ 
 \text{and}\\
  E(x,t)&=\int_{\{\tau\geq t\}}\lvert e^{-\lambda\tau}\phi(B_{\tau})-\phi(z)\rvert\,d\mathbb{P}_x.
 \end{align*}
 For the estimate on $C(x,t)$ we 
 observe that on $\{ \tau<t,\tau<\tau_{z,\delta}\}$ 
 \begin{align*}
 \mathbb{P}_x(B_{\tau}\in B(z,\delta)\cap \partial G)=1\quad\forall x\in B(z,\delta)\cap G
 \end{align*}
 and therefore
 \begin{align*}
 \mathbb{P}_{x}(\lvert\phi(B_{\tau})-\phi(z)\rvert\leq \epsilon)=1\quad\forall x\in B(z,\delta)\cap G.
 \end{align*}
 Furthermore, on $\{ \tau<t,\tau<\tau_{z,\delta}\}$ we have for $t>0$ small enough $\lvert e^{-\tau \lambda}-1\rvert\leq 2\lvert\lambda\rvert t$, since $\tau<t$. 
 Hence for $x\in B(z,\delta)\cap G$, we get:
 \begin{align*}
 C(x,t)&\leq \int_{\{ \tau<t,\tau<\tau_{z,\delta}\}}\lvert e^{\lambda \tau}\rvert \cdot\lvert\phi(B_{\tau})-\phi(z)\rvert+\lvert e^{-\lambda\tau}-1\rvert\cdot\lvert \phi(z)\rvert \,d\mathbb{P}_{x}\\
&\leq e^{-\lambda_{1}t}\cdot\epsilon+2\lvert\lambda\rvert t\cdot M\\
&\leq C\cdot \epsilon
\end{align*}
 for all $t\in(0,t_{1}]$, where $t_{1}$ and $C$ are positive constants  independent of $x$.\\
 For the estimate on $D(x,t)$, first note that for $x\in B(z,\delta/2)\cap G$ we have $B(x,\delta/2)\subset B(z,\delta)$. Therefore a path starting at $x$ must first exit $B(x,\delta)$ before it exits $B(z,\delta)$.
 Hence
 \begin{align*}
\mathbb{P}_{x}(\tau_{x,\delta/2}\leq \tau_{z,\delta})=1.
\end{align*}
 This gives us for $x\in B(z,\delta/2):$
\begin{align}\label{eq:P}
\mathbb{P}_{x}(\tau\geq\tau_{z,\delta})&=\mathbb{P}_{x}(\tau\geq\tau_{z,\delta},\tau<t)+\mathbb{P}_{x}(\tau\geq\tau_{z,\delta},\tau\geq t)\\
&\leq \mathbb{P}_{x}(\tau_{x,\delta/2}<t)+\mathbb{P}_{x}(\tau\geq t).
\end{align}
Now from Lemma \ref{lemma:thmexit2} it follows that there is a $t_{2}$ such that for all $t\in (0,t_{2}]$ we have
\begin{align*}
\mathbb{P}_{x}(\tau_{x,\delta/2}<t)\leq \epsilon\quad \forall x\in B(z,\delta/2)\cap G.
\end{align*}
Where we choose $t_{2}$ such that $t_{2}\leq t_{1}$ and $e^{\lvert\lambda\rvert t_{2}}\leq 2$.
Fix this choice of $t_2$, then from Lemma \ref{lemma:thmexit1} we can choose $\delta_{1}\in(0,\delta/2)$ such that:
\begin{align*}
\mathbb{P}_{x}(\tau\geq t_{2})\leq \epsilon\quad\forall B(z,\delta_{1})\cap G.
\end{align*}
With this choice of $\delta_{1}$ and $t_{2}$ together with equation (\ref{eq:P}) we get for $x\in B(z,\delta_{1})\cap G$ and $t\in (0,t_{2})$:
\begin{align*}
D(x,t_{2})&\leq\int_{ \{\tau <t_2,\tau\geq \tau_{z,\delta}\} }\lvert e^{-\lambda\tau}\phi(B_{\tau})-\phi(z)\rvert\,d\mathbb{P}_x\\ 
&\leq (2\cdot M+M)\mathbb{P}_{x}(\tau\geq\tau_{z,\delta})\\
&\leq 3M(\epsilon+\epsilon).
\end{align*}
For the estimate of $E(x,t)$ we observe that, due to Lemma \ref{lemma:thmexit1} there is a $\delta_{2}\in(0,\delta_{1})$ such that 
\begin{align*}
\int_{\{\tau\geq t_{2}\}} \lvert e^{-\lambda\tau} \rvert \,d\mathbb{P}_{x}<\frac{\epsilon}{M}\quad\forall x\in B(z,\delta_{2})\cap G
\end{align*}
and 
\begin{align*}
\mathbb{P}_{x}(\tau\geq t_{2})\leq \frac{\epsilon}{M}\quad\forall x\in B(z,\delta_{2})\cap G.
\end{align*}
Therefore we get for $B(z,\delta_{2})\cap G$ 
\begin{align*}
E(x,t_{2})&\leq \int_{\{\tau\geq t_{2}\}}\lvert e^{-\lambda \tau}\phi(B_{\tau})\rvert+\lvert \phi(z)\rvert\,d\mathbb{P}_{x}\\
&\leq M\cdot  \int_{\{\tau\geq t_{2}\}}\lvert e^{-\lambda \tau}\rvert +M\cdot \mathbb{P}_{x}(\tau\geq t_{2})\\
&\leq 2\epsilon.
\end{align*}
Now we can conclude the proof by keeping our choices of $t_{2},\delta_{1}$ and $\delta_{2}$ and combine the estimates above to
\begin{align*}
\lvert \psi(x)-\phi(z)\rvert&\leq C(x,t_{2})+D(x,t_{2})+E(x,t_{2})\\
&\leq (C+6M+2)\cdot\epsilon.
\end{align*}
Therefor $\psi(x)\to \phi(z)$ as $x\to z$ for all $x\in B(z,\delta_{2})\cap G$.
 \end{proof}
 This finishes the proof of Theorem \ref{thm:exitbrown}.

\footnotesize
%
\bibliography{literature}
\bibliographystyle{alpha} 






\end{document}